\documentclass[11pt]{article}
\usepackage[margin=1.1in]{geometry}
\usepackage{enumerate}
\usepackage{amssymb}
\usepackage{amsmath}
\usepackage{mathrsfs}
\usepackage{amsthm}
\usepackage{mathtools}
\usepackage{float}
\usepackage{braket}
\usepackage{cancel}
\usepackage{eufrak}
\usepackage{xcolor}

\usepackage{multirow}
\usepackage{tikz}
\usetikzlibrary{arrows.meta,bending,positioning, decorations.pathmorphing}  

\usepackage[
backend=biber,
sorting=nyt
]{biblatex}

\usepackage{comment}

\usepackage{cleveref}
\newtheorem*{theorem*}{Theorem}
\newtheorem{theorem}{Theorem}[section]

\newtheorem{proposition}[theorem]{Proposition}
\newtheorem{corollary}[theorem]{Corollary}
\newtheorem{lemma}[theorem]{Lemma}
\newtheorem{definition}[theorem]{Definition}

\newtheorem{conj}[theorem]{Conjecture}
\newtheorem{rem}[theorem]{Remark}
\author{Aya Bernstine\thanks{School of Computer Science and Engineering, Hebrew University, Jerusalem 91904, Israel. e-mail: aya.bernstine@mail.huji.ac.il.}
			\and{Nati Linial\thanks{School of Computer Science and Engineering, Hebrew University, Jerusalem 91904, Israel. e-mail: nati@cs.huji.ac.il. Supported in part by a NSF-BSF research grant "Global Geometry of Graphs".}}}
\date{}
\title{An approach to the girth problem in cubic graphs}
\begin{document}

\maketitle

\begin{abstract}
We offer a new, gradual approach to the {\em largest girth problem for cubic graphs}. It is easily observed that
the largest possible girth of all $n$-vertex cubic graphs is attained by a {\em $2$-connected} graph $G=(V,E)$.
By Petersen's graph theorem, $E$ is the disjoint union of a $2$-factor and a perfect matching $M$. We refer to the edges 
of $M$ as {\em chords} and classify the cycles in $G$ by their number of chords. We define $\gamma_k(n)$ to be the largest
integer $g$ such that every cubic $n$-vertex graph with a given perfect matching $M$ has a cycle of length at
most $g$ with at most $k$ chords. Here we determine this function up to small additive constant for $k= 1, 2$ and
up to a small multiplicative constant for larger $k$.

\end{abstract}

\section{Introduction}
The {\em girth} of a graph $G$ is the shortest length of a cycle in $G$. 
Our main concern here is with bounds on $g(n)$, the
largest girth of a cubic $n$-vertex graph. Namely, we
seek the best statement of the form "Every $n$-vertex cubic graphs must have a short cycle". An
elementary existential argument \cite{Erd1963RegukreGG} shows that $g(n)\ge \log_2 n$
for infinitely many values of $n$. The best lower bound that we currently have 
is  $g(n)>\frac{4}{3}\log_2 n -2 $. It is attained by sextet graphs, an infinite family of
explicitly constructed graphs, first introduced in \cite{Sextet_graphs},
the girth of which
was determined in \cite{girths_Sextet_graphs}. 
On the other hand, the only upper bound that we have is
the more-or-less trivial {\em Moore's Bound} which says that
\[(2+o_n(1))\log_2 n \ge g(n).
\]
It is our belief that in fact a better upper bound applies
\begin{conj}
There is a constant $\epsilon_0>0$ such that $(2-\epsilon_0)\log_2 n \ge g(n)$
\end{conj}
Concretely, Moore's lower bound on $n$, the
number of vertices in a $d$-regular graph of girth $g$ is as follows. For odd
$g=2k+1$ it is $n\ge 1+d\sum_{i=0}^{k-1}(d-1)^i$, for even $g=2k$ it is 
$n\ge 1+(d-1)^{k-1}+d\sum_{i=0}^{k-2}(d-1)^i$. 

{\em Moore Graphs} which satisfy this with equality were fully characterized about a half century ago
\cite{damerell_1973}, \cite{Bannai1973OnFM}. 
There is an ongoing research effort to determine, or at least estimate
$g(n)$, and we initiate here a new approach to this problem.
More generally, the search is on for {\em $(d,g)$-cages}.
These are $d$-regular graphs of girth $g$ with the smallest possible number of vertices.
The record of the best known bounds in this area is kept in \cite{daynamic_survey:2013}.

\section{Statement of the problem}
As the following observation shows, it suffices to consider $2$-connected graphs. 
We give the simple proof for completeness sake, but actually
more is known about the connectivity of cubic cages \cite{odd_g_cage_are_k_connected}.
\begin{proposition}
Among all cubic $n$-vertex graphs that have the largest possible girth $g(n)$,
at least one graph is $2$-connected.
\end{proposition}
\begin{proof}
It is well-known and easy to show that a cubic graph is $2$-connected iff it is bridgeless. 
So let the edge $x y\in E$ be a bridge in $G=(V,E)$, a cubic $n$-vertex graph.
Let $x'\neq y$ be a neighbors of $x$, and $y'\neq x$ a neighbor of $y$. Let $G'$ be obtained
from $G$ be deleting the edges $x x'$ and $y y'$ and adding the new edges $x y'$ and $y x'$. It is easily verified that $G'$ has fewer bridges than $G$, while $\text{girth}(G')\ge \text{girth}(G).$
\end{proof}
We next recall Petersen's theorem \cite{Petersen_theorem}: The edge set of every $2$-connected cubic graph $G$ can be decomposed into a perfect matching $M$ and a $2$-factor $C$. We
call this a {\em Petersen decomposition} of $G$, and refer to the edges of $M$ as {\em chords}.

The following is the main focus of this paper: 

\begin{definition}
$\gamma_k(n)$ is the smallest integer $g$ such every cubic $n$-vertex graph with any Petersen Decomposition has a cycle of length $\le g$ with at most $k$ chords.
\end{definition}

We start with some elementary observations concerning $\gamma_k(n)$.

\begin{itemize}
\item 
$\gamma_k(n)$ is a non-increasing function of $k$.
\item
In a very loose sense, for fixed $k$, the function $\gamma_k(n)$ is non-decreasing with $n$, 
although this is not literally correct. For example, the Tutte-Coxeter Graph, a well-known Moore graph
shows that $\gamma(30)=8$. Also, clearly, $\gamma_3(30)\ge\gamma(30)$.
On the other hand, extensive computer searches
show that $\gamma_3(32)=7$.
\item
$\gamma_0(n)=n$.
\item
$\gamma_1(n)=\lfloor\frac n2\rfloor+1$. The unique extremal
example is a $(2g-2)$-cycle $C$ and a matching $M$ that
matches every antipodal pair of vertices in $C$. 
\item
A lower bound on $\gamma_k(n)$ amounts to a construction of a cubic graph along with a Petersen Decomposition, where every short cycle has more than $k$ chords.
\item
Upper bounds on $\gamma_k(n)$ are proofs that every cubic graph
on $n$ vertices with a given Petersen Decomposition must have a short cycle with few chords.
\item
For $k\ge \log_2 n$ there holds $\gamma_k(n)=g(n)$.
\end{itemize}

\subsection{Paper Organization}
The rest of the paper is organized as follows. In Section \ref{s:results} we state our
upper and lower bounds on $\gamma_k(n)$.
In Sections \ref{s:upper_bounds} and \ref{s:lower_bounds} we
prove the upper resp.\ lower bounds. 
In Section \ref{s:genaral_d} we make some connection
with the girth problem in its general form. 
\section{Our results}\label{s:results}
In this paper we present an upper and lower bounds on cycles with at most $k$ chords
\begin{theorem}\label{th:main}
Let $n$ be an even number, then
\begin{enumerate}
\item\label{part1} $\sqrt{2n}- \frac 52< \gamma_2(n)\le \sqrt{2n}+2$. In fact, $\sqrt{2n}+\frac 12<\gamma_2(n)$\label{thm:k=2}
holds for infinitely many $n$'s.
\item\label{part2}
\[ \frac 12 (2n+\frac {9}{4})^{\frac 12}+\frac 54\le \gamma_3(n)\le (2n)^{\frac 12}+1 \]\label{thm:k=3}
\item
\[ (2n)^{1/3}+O(1)  \le \gamma_5(n) \le \gamma_4(n) \le \frac{3}{2}(2n)^{1/3}\]

\label{thm:k=4}
\item
\[ 2(\frac{n}{4})^{1/4}+O(1) \le \gamma_7(n) ~~and~~2(\frac{n}{4})^{1/6}+O(1) \le \gamma_{11}(n)\]
\item 
\[ \gamma_{2l}(n)\le  3l+\frac{1}{2}(l+1+\ln 2)(n^{\frac{1}{l+1}}) \] for every $l\ge 3$.
\item
\[ n^{\frac{4}{3q}}+O(1)\le \gamma_q(n) \] for every prime power $q$.
\end{enumerate}

\end{theorem}

\section{Upper bounds - proofs}\label{s:upper_bounds}
We prove first the upper bounds in Items \ref{part1} and \ref{part2},
starting with some necessary preparations. Let $\alpha_1,\ldots,\alpha_t$
be a sequence of $t\ge 2$ distinct positive integers.
The corresponding {\em cycle} $R$ is comprised of
the sequence $(\alpha_1,\ldots,\alpha_t)$ and its cyclic shifts
$(\alpha_j,\ldots,\alpha_t, \alpha_1,\ldots,\alpha_{j-1})$, for every $j=2,\dots,t$. We say,
for $i=1,\ldots,t-1$ that $\alpha_i, \alpha_{i+1}$ are {\em adjacent} in $R$. Also $\alpha_t, \alpha_{1}$ are adjacent in $R$. We denote
the set of (unordered) adjacent pairs in $R$ by $A(R)$. Finally we define $w(R):=\sum_{\alpha,\beta\in A(R)} |\alpha-\beta|$.
\begin{proposition}\label{prop:number_premutatuions} 
Let $\Pi=\{R_1\ldots R_l\}$ be a partition of $[k]$ into cycles
for some even integer $k$. Then  $\sum_{i}w(R_i)\leq \frac{k^2}{2}$. 
The bound is tight.
\end{proposition}
\begin{proof}
We say that a cycle $R$ {\em covers} the interval $[i,i+1]$ if $i, i+1$ lie between $\alpha$ and $\beta$
for some adjacent pair $\alpha,\beta\in A(R)$ (allowing $\{i,i+1\}$
and $\{\alpha,\beta\}$ to have non-empty intersection).
Let $G_i$ be the number of cycles in $\Pi$ that cover the interval $[i,i+1]$. Clearly 
\[\sum_{j}w(R_j) =\sum |G_i|.\]
We turn to bound $\sum |G_i|$.
For $i<\frac{k}{2}$, the adjacent pair $\alpha, \beta$ can cover the interval $[i,i+1]$
only if $\min (\alpha, \beta) \le i$. Every index smaller than or equal to $i$
can contribute a total of $2$ to $|G_i|$, once moving left and once moving right along the cycle. 
Consequently $|G_i|\le 2i$. By symmetry, 
the same argument yields $|G_i|\le 2(k-i)$ when $i>\frac{k}{2}$, and finally
$|G_\frac{k}{2}|\leq k$. We sum these inequalities and conclude that
$$\sum_{i=1}^{l} |G_i| \leq \frac{k(k-2)}{4}+\frac{k(k-2)}{4}+k=\frac{k^2}{2}.$$
Equality is attained e.g., for the partition $R_1\ldots R_{\frac{k}{2}}$ into $\frac{k}{2}$ cycles with $R_i=\{i,i+\frac{k}{2}\}$.
\end{proof}
The following lemma proves the upper bound on $\gamma_2(n)$ in Theorem \ref{thm:k=2}.

\begin{lemma}\label{2_chords_uppe_bound}
Let $G=(V,E)$ be an $n$-vertex $2$-connected cubic graph with Petersen decomposition $E=M\sqcup C$. Then $G$ has a cycle of length $\le \sqrt{2n}+2$ with at most $2$ chords. 
\end{lemma}
\begin{proof}
Let us mention that actually there always exists
a cycle of length $\le \sqrt{2n}+1$ with at most $2$ chords.
We omit the somewhat laborious details of this stronger argument,
and turn to prove the theorem.\\

If $uv\in M$ is a chord, we write $v^*=u$.
Let $A$ be an arc of length $k\leq \sqrt{2n}$ 
along a cycle $\sigma_0$ in $C$, and let $A^* := \{v^*|v\in A\}$.
If $A^*\cap A \neq \emptyset$ this yields a cycle of length $\le \sqrt{2n}+1$ with only one chord. So,
we may and will assume that $A^*\cap A = \emptyset$.

Let $\sigma\neq\sigma_0$ be another cycle of $C$.
If $|A^*\cap\sigma|=r\ge 2$, let us denote the vertices of 
$A^*\cap\sigma$ by $u_1^*,\ldots,u_r^*$ in cyclic order and with indices taken $\bmod ~r$.
For every $i=1,\ldots,r$ we define the following $2$-chord cycles:
It starts with the chord $u_i, u_i^*$; along $\sigma$
to $u_{i+1}^*$; the chord $u_{i+1}^*, u_{i+1}$; finally along $A$ to $u_i$.\\
When $r=1$ we consider instead the chord-free cycle $\sigma$.\\
This construction applies as well to $\sigma=\sigma_0$, except that the above $2$-chord cycles
are defined only for $i=1,\ldots,r-1$. 

We have mentioned in total either 
$k-1$ or $k$ such $2$-chord cycles depending on whether or not $A^*$ intersects with $\sigma_0$. 

We turn to bound the total length of those $2$-chord cycles, 
starting with their parts that traverse $A$.
The collection of $2$-chord cycles that correspond to any cycle $\sigma$ in $C$ induces a cycle on $A$.
As we go over all $\sigma$
we obtain a partition of $A$ into cycles. By Proposition \ref{prop:number_premutatuions}
the sum total of their lengths does not exceed $\frac{k^2}{2}$. Other than their overlaps
on $A$, the $2$-chord cycles traverse every chord twice and are otherwise overlap-free.

Therefore their total length is at most $\le n+\frac{k^2}{2}+{k-1}$. Consequently,
the average length of such cycle does not exceed
$$\frac{n+\frac{k^2}{2}}{k-1}+1.$$
We optimize and take $k=\sqrt{2n+1}+1$ to conclude that at least one
of these $2$-chord cycles with length  $\le \sqrt{2n}+2$.

\end{proof}

We prove next the upper bound on $\gamma_4(n)$.
\begin{lemma}\label{lem:four_chord_upper}
Let $E=M\sqcup C$ be a Petersen decomposition of an $n$-vertex $2$-connected cubic graph $G=(V,E)$. There is a cycle in $G$ of length $\leq L=\lambda n^{\frac{1}{3}}+3$ with at most four chords, where $\lambda=3\cdot 2^{-\frac{2}{3}}\simeq 1.89$.
\end{lemma}
\begin{proof}
Arguing by contradiction, we assume  
that all cycles of $C$ are longer than $L$.
We measure distances along $C$. Thus, the $r$-{\em neighborhood} of $v\in V$, is the arc of $2r+1$ vertices centered at $v$ in the cycle of $C$ to which $v$ belongs.
The $r$-neighborhood of $S\subset V$ is the union of the $r$-neighborhoods of all vertices in $S$.  

Let $A_0$ be an arc of length $k\leq \frac{2L}{3}-2$ along a cycle of $C$. We may assume that any two vertices in $A_0^*$ are at distance at least $\ge\frac{L}{3}-1$, 
for otherwise we obtain a cycle of length $<L$ with only two chords.
Let $r=\frac{L}{6}-\frac{1}{2}$.
It follows that $A_1$, the $r$-neighborhood of $A_0^*$ has cardinality $|A_1| = 2kr$. 
Of course $|A_1^*|=|A_1|$. 

We say that $u, v\in A_1^*$ are {\em consecutive} if the shortest arc of $C$ between 
them contains no additional vertices of $A_1^*$. We number the vertices of $A_1^*$ from $1$ to $2rk$
cyclically along $C$, proceeding from a vertex to the
consecutive one. If $u\in A_1^*$ is the $i$-th vertex in this numbering,
we associate with it the following $2$-chords path $P_i$ which starts at $u$ and ends in $A_0$.
This path begins with a hop from $u$ to $u^*\in A_1$. Then comes a walk along
$C$ to the closest vertex in $A_0^*$. Call this closest vertex $\beta_i$ and the distance traveled $b_i$. 
The path ends with a hop from $\beta_i$ to $\beta_i^*\in A_0$.
 
Let $v\in A_1^*$ be the $(i+1)$-st in this order. Consider the following $4$-chord cycle: 
\begin{enumerate}
    \item from $u$ to $v$ along $C$ a distance of $a_i$.
    \item along $P_{i+1}$ to $\beta_{i+1}^*\in A_0$ .
    \item a walk through $A_0$ to $\beta_{i}^*$ a distance of $c_i$.
    \item from ${\beta_{i}^*}$ we traverse $P_{i}$ to $u$, in the direction
    opposite to the above description. 
\end{enumerate}

The length of this cycle is thus $a_i+b_i+c_i+b_{i+1}+4$. There are $|A_1^*|-1=2kr-1$ cycles in this list,
since we exclude the case where $u$ is the $2rk$-th vertex of $A_1^*$ and $v$ is the first. 
The total length of these cycles is
$$\sum_{i=1}^{2rk-1}(a_i+b_i+c_i+b_{i+1}+4)$$
\begin{figure}[!tp]\label{fig:full_illustration}
  \centering
  \includegraphics[scale=0.3]{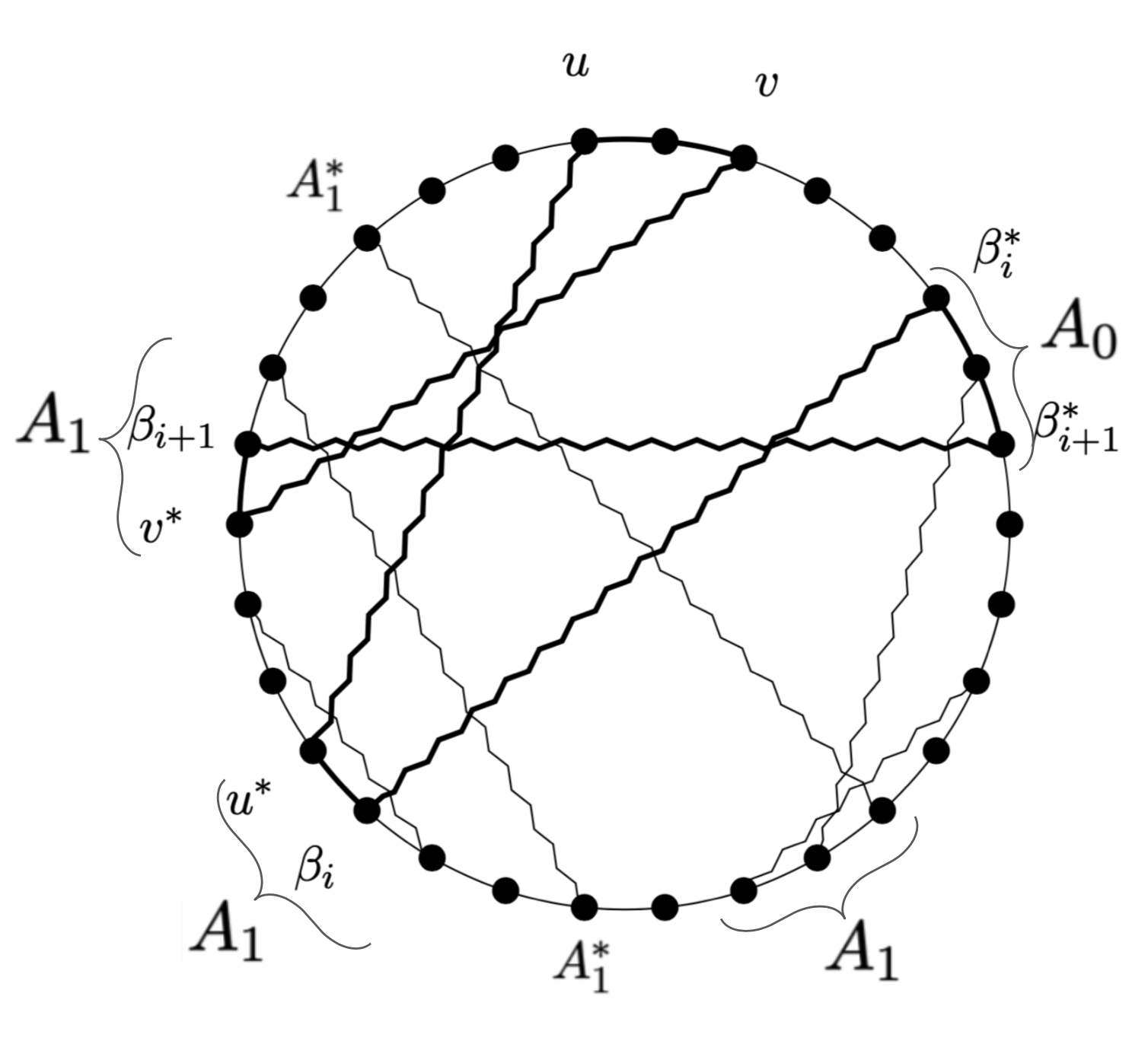}
  \caption{An illustration of the expansion of $A_0$. Zigzag edges represent chords.}
\end{figure}
We turn to bound the different parts of this sum.
The stretches of $a_i$ steps are disjoint from
$A_0,A_1$, and $A_0^*$ hence $\sum_{i=1}^{2rk-1}a_i\leq n-|A_0|-|A_0^*|-|A_1| =n-2k-2kr$. Consider
next the stretches of walks that go from a vertex in $A_1$ to the closest vertex in $A_0^*$.
For every $\beta_{i}\in A_0^*$ the sum of the distances of all vertices
in the $r$-neighborhood of $\beta_i$ is $r^2$.
It follows that $\sum_{i=1}^{2kr-1}b_i+\sum_{i=2}^{2kr}b_i \le 2\sum_{i=1}^{2kr} b_i=2k r(r+1)$.
Finally we turn to bound $\sum c_i$. We argue here in the same way that we did in the proof of 
Lemma \ref{2_chords_uppe_bound}, except that here the argument is carried out $2r$ times.
Consequently, $\sum_{i=1}^{2rk-1} c_i\leq 2r(\frac{k^2}{2})$.

Thus the sum total of the cycles' lengths is at most $n-2k-2kr+2k\cdot r(r+1)+rk^2+4(2kr)$.
This gives the following upper bound on their average length
$\leq \frac{n}{2rk-1}+r+\frac{k}{2}+4$.
We get the tightest upper bound on the length by letting $k=\sqrt{\frac{n}{r}}$ 
and $r=\frac{{(2n)}^{1/3}}{2}$. We conclude that there is at least one cycle of length
$\le 3(\frac{n}{4})^\frac{1}{3}+4$.
\end{proof}
We finally present an upper bound on $\gamma_k(n)$ for
every even $k$. This yields an upper bound also for the case of odd $k$,
since $\gamma_k(n)$ is a decreasing function of $k$. 
The proof is just an adaptation of the upper bound proof for $\gamma_4(n)$.

\begin{lemma}
Let $E=M\sqcup C$ be a Petersen decomposition of an $n$-vertex $2$-connected cubic graph $G=(V,E)$. There is a cycle in $G$ of length $\le 3l+(l-1)(2^{-\frac{l}{l+1}}n^{\frac{1}{l+1}})+n^{\frac{1}{l+1}}2^{\frac{1}{l+1}}$ with at most $2l$ chords.
\end{lemma}\label{lm:lower_bound_on_k}
\begin{proof}
We only sketch the proof, since it is similar to the proof of lemma \ref{lem:four_chord_upper}.
Let $A_0$ be an arc of length $k<\frac{1}{2}(l+1)(n^{\frac{1}{l+1}})$. For every $0<i\le l-1$, we
let $A_i$ be the $r$-neighborhood of $A_{i-1}^*$. It follows that $A_{i}$ has cardinality $|A_i|=(2r)^{i}k$.
We number the vertices of $A_{l-1}^*$ from $1$ to $(2r)^{l-1}k$ cyclically along $C$. 
Let $u\in A_{l-1}^*$ be the $i$-th vertex in this numbering.
We associate with it the following $l$-chords path $P_{i}$ which starts at $u$ and ends in $A_0$. 
It begins with a hop from $u$ to $u^*\in A_{l-1}$ then a walk along $C$ to the closest vertex in $A_{l-2}^*$ call it $\beta_{l-2}$ a distance of $b^i_{l-2}$ then a hop to $\beta_{l-2}^*$ and in general from vertex $\beta_{j}^*\in A_{j-1}$ a walk along $C$ to the closest a distance $b_i^{j}$ to a vertex $\beta_{j-1}\in A^*_{j-2}$. 
Let $v\in A_{l-1}^*$ be the ($i+1$)-st in this order and consider the following $2l$-chord cycle:
\begin{enumerate}
    \item from $u$ to $v$ along $C$ a distance $a_i$.
    \item along $P_{i+1}$ to $\beta^{i+1*}\in A_0$.
    \item a walk along $A_0$ to $\beta^{i*}_{0}$ distance $c_i$.
    \item from $\beta^{i*}_{0}$ we traverse $P_i$ to $u$ in the opposite direction of the path.
\end{enumerate}
The length of this cycle is thus $a_i+\sum_{j=1}^{l-1} b^{j}_{i}+b^{j}_{i+1}+c_{i}+2l$. There are $|A_{l-1}|^*-1=(2r)^{l-1}k-1$ cycles to consider. 
The total length is  $\sum^{(2r)^{l-1}k-1}_{i=1}(a_i+c_i+\sum_{j=1}^{l-1} b^{j}_{i}+b^{j}_{i+1}+2l)$, 
we bound all the steps as we did in proving Lemma \ref{lem:four_chord_upper}. This
yields that $\frac{n}{(2r)^{l-1}k-1}+(l-1)r+\frac{k}{2}+3l-2$
is an upper bound on their average length, and this for any two integers $r,k\ge 1$. 
The optimal choice is $k=\sqrt{\frac{2n}{(2r)^{l-1}}}$ and $r=\frac{(2n)^{\frac{1}{l+1}}}{2}$. 
With this choice we conclude that there is at least one cycle of length $\le 3l+(l-1)(2^{-\frac{l}{l+1}}n^{\frac{1}{l+1}})+n^{\frac{1}{l+1}}2^{\frac{1}{l+1}}$
\end{proof}

\section{Lower bounds}\label{s:lower_bounds}
We prove the lower bound on $\gamma_2(n)$ 
by providing an explicit construction. In the Petersen Decomposition of our
graph, $C$ is a Hamilton cycle.

\begin{lemma}\label{lem:construction_with_2_chords}
$\gamma_2(n) \ge\sqrt{(2n+\frac 94)}+\frac 12$
whenever $n=8l^2+6l$ with $l$ a positive integer. Namely, we construct
a cubic $n$-vertex graph $G=(V,E)$ with a Petersen Partition 
$E=M\sqcup C$, where $C$ is a Hamilton cycle and
where the shortest cycle with $2$ chords or less has length $4l+2$.
\end{lemma}
\begin{proof}
All the indices mentioned below are taken $\bmod ~n$. The vertices of
$C$ are numbered $0,\ldots, n-1$ in this order.  
Note that $n$ is divisible by $2l$. We divide $C$ into $n/2l$
equal-length {\em blocks} and construct
a graph that is invariant under rotation by $2l$. The first block is comprised of vertices $0,\ldots,2l-1$.
Let $A = (a_0,a_1...a_{l-1})$ be a sequence of $l$ odd integers. For every $l>j\ge 0$, we introduce the chord $(2j, 2j+a_j)$, and, as mentioned, we rotate these chords with steps of $2l$. Stated differently, we (uniquely) express every even integer $x$ in the range $0,\ldots,n-1$ as $x=2lk+2j$, where $
0 \le k<\frac{n}{2l}$, $0\le j<l$ and connect vertex $x$ by a chord to $x+a_j$. 
Since the integers $a_{i}$ are all odd, every chord connects a vertex of even index to one of odd index. Also, every even vertex is incident with exactly one chord. In order for $G$ to be cubic, also every odd vertex must be incident with a single chord. Namely, we need to choose the $a_i$ so that
\begin{equation}\label{eq:simple_matching}
2lk+2j+a_j=2lr+2s+a_s \Rightarrow k=r, j=s \text{~for every~}0 \le k, r<\frac{n}{2l},~ 0\le j, s<l.
\end{equation}

In order that every cycle with exactly one chord has length $\geq g = 4l+2$ we must satisfy
\begin{equation}\label{eq:cycle_one_chord}n-g+1\geq a_j\geq g-1 \text{~for every~} 0\leq j \leq l-1 \end{equation} 

We choose the parameters so that $a_i \equiv a_j \bmod 2l$ for all $i,j=0,\ldots,l-1$. This yields
Condition (\ref{eq:simple_matching}) by considering the equations $\bmod~ 2l$.
Concretely, we let $A$ be the arithmetic progression $a_j = 4l(j+1)+1$.
Condition (\ref{eq:cycle_one_chord}) is easily seen to hold.

Consider the shortest $2$-chord cycle with chords
\[(2kl+2j, 2kl+2j+a_j) \text{~and~} (2rl+2s, 2rl+2s+a_s).\]
We need to show that its length is $\ge g$. There are two cases to consider, as the arc of this cycle
which starts at $2kl+2j$ can end at either $2rl+2s$ or $2rl+2s+a_s$. Let us introduce the notation $\|x\|:=\min \{x,n-x\}$ for an integer $n-1\ge x \ge 0$. We extend the definition to all integers $x$ by first taking the residue of $x \bmod n$. Using this terminology,
we need to show that if $(k,j)\neq (r,s)$, then
\[\|2kl+2j-(2rl+2s)\|+\|2kl+2j+a_j-(2rl+2s+a_s)\|\ge g-2\]
and
\[\|2kl+2j-(2rl+2s+a_s)\|+\|2kl+2j+a_j-(2rl+2s)\|\ge g-2.\]

Due to the cyclic symmetry of our construction no generality is lost if we assume that $r=0$ and $0\le k < 2l+2$. We also spell out the values of $a_j, a_s$ and $g$ and now we have to show that if $k\neq 0$ or $j\neq s$, then
\begin{equation}\label{eq:cycle_two_even_chords}\|2kl+2(j-s)\|+\|2kl+(4l-2)(j-s)\|\ge 4l \end{equation} 
and
\begin{equation}\label{eq:cycle_two_chords_odd_even}\|2l(k-2s-2)+2(j-s)-1\|+\|2l(k+2j+2)+2(j-s)+1\|\ge 4l.\end{equation}
The inequality $\|x\|\ge y$ for an integer $n>x\ge 0$ and a positive integer $y$ is equivalent to the conjunction of the inequalities $x\ge y$ and $n-x \ge y$. For negative arguments we use the fact that $\|-x\|=\|x\|$, then apply the above. Let us refer to the left hand side of Equation (\ref{eq:cycle_two_even_chords}) as $\|A_1\|+\|A_2\|$ and to Equation (\ref{eq:cycle_two_chords_odd_even}) as $\|B_1\|+\|B_2\|$. Note that $n-B_1\ge n-A_1\ge 4l$. It suffices therefore to prove that $|A_1|+\|A_2\|\ge 4l$ and $|B_1|+\|B_2\|\ge 4l$. 

We start with the case $k=0$, which implies $j\neq s$. In Equation (\ref{eq:cycle_two_even_chords}) this yields $|A_1|\ge 2$ and $n-4l>4l^2-6l+2\ge |A_2|\ge 4l-2$ hence $|A_1|+\|A_2\|\ge 4l$. In Equation (\ref{eq:cycle_two_chords_odd_even}) when $k=0$ either $j\ge 1$ and then $n-4l>4l^2+2l-1>|B_2|\ge 4l$ or $s\ge 1$ and then $|B_1|\ge 4l$ hence $|B_1|+\|B_2\|\ge 4l$. 

We turn our attention to Equation (\ref{eq:cycle_two_even_chords}) with $k\ge 1$. 
It is easy to see that $A_1>4l$ when $k\ge 3$. On the other hand, $A_2>n-4l$ when $k\leq 2$. It remains to show that $|A_1|+|A_2|\ge 4l$ when $k\in\{1,2\}$. This is done in the following case analysis.
\begin{itemize}
\item 
There holds $A_1<0$ exactly when $s>kl+j$. But then $A_2<0$ as well and we get $|A_1|+|A_2|=-A_1-A_2 = 4l(s-k-j)$.
Clearly $s>kl+j\ge k+j$ so that $|A_1|+|A_2|\ge 4l$ as claimed.
\item
When $A_1>0>A_2$, there holds $|A_1|+|A_2|=A_1-A_2=(4l-4)(s-j)$. We are in this range exactly when when $\frac{kl}{2l-1}<s-j<kl$, in this case we would compute and $A_1-A_2\geq 4l$, because $j+1<s$, and $k\ge 1$. 
\item
In the last remaining case $A_1,A_2>0$. Here $s<\frac{kl}{2l-1}+j$ and $A_1+A_2 =4l(k+j-s)>4l$.
\end{itemize}
Inequality  (\ref{eq:cycle_two_even_chords}) follows.

We proceed to Equation (\ref{eq:cycle_two_chords_odd_even}). Note that $B_2>0$. Also, $B_2\le 4l$ only when $k=j=0$ in this case $|B_1| = |-4ls-4l-2s-1|>4l$. Hence we are left to consider the case where $\|B_2\|=n-B_2$, the only case when $n-B_2<4l$ is when $k=2l+1$ and $j=l-1$ but then $|B_1|>4l$. 
Inequality  (\ref{eq:cycle_two_chords_odd_even}) follows.
With this we can conclude that the construction is valid.

\end{proof}

\begin{rem}
For $l=1$, this construction yields the Heawood graph which is a Moore graph.
\end{rem}

We proceed to bound $\gamma_3(n)$ by giving an appropriate explicit construction.

The main idea of the construction is to transform the graph from Lemma \ref{lem:construction_with_2_chords} into a "bipartite" version. 
This relies on the fact that in the original construction $n$ is even and every chord connects a vertex of even index to one with an odd index. 
In this construction no cycle has exactly $3$ chords, while the length of every $2$-chord cycle in the original graph is only cut in half at worst.

\begin{lemma}
$\gamma_3(n) \ge\frac 12\sqrt{(2n+\frac {9}{4})}+\frac 54$
for every $n=8l^2+6l$ with $l$ a positive integer. Namely, we construct
a cubic $n$-vertex graph $G=(V,E)$ with a Petersen Partition, where the shortest cycle 
with $3$ chords or less has length $2l+2$.
\end{lemma}
\begin{proof}
The graph that we construct has the same vertex set $V$ and the same set of chords $M$ as in the graph $G$
of Lemma \ref{lem:construction_with_2_chords}. What changes is the $2$-factor $C$ of the Petersen Partition.
In the original construction $C$ was a Hamilton Cycle that traverses the vertices in order. 
In the present construction $C$ is the disjoint union of two cycles, $C_{even}$ and $C_{odd}$
that traverse all the even-indexed resp.\ odd vertices in order.

It is clear that $H$ is cubic, since every vertex is incident to exactly one chord and has two neighbors in the $2$-factor in which it resides.
As shown in Lemma \ref{lem:construction_with_2_chords}, every chord connects an even-indexed vertex to an odd one. 
It follows that every cycle in $H$ has an even number of chords.
Therefore we only need to consider cycles with $2$ chords. 
Also, the only chordless cycles are $C_{odd},C_{even}$ whose length is $\frac{n}{2}\gg 2l+2$. 

Finally we come to cycles $\cal{C}$ with exactly $2$ chords in $H$. Such a cycle is composed of an arc from $v_1$ to $v_2$ in $C_{odd}$, a chord $v_2 v_3$ to $v_3\in C_{even}$, an arc
from $v_3$ to $v_4$ in $C_{even}$ and a chord $v_4 v_1$ back to $C_{odd}$.
We can consider the cycle $\Tilde{\cal{C}}$ in $G$ that traverses $C$ from $v_1$ to $v_2$, takes the chord $v_2 v_3$, then traverses $C$ from $v_3$ to $v_4$ and finally the chord $v_4 v_1$.
This is a $2$-chord cycle in $G$. and 
$\text{length}(\tilde{\cal{C}})=2\cdot\text{length}(\cal{C})\rm-2$. 
As Lemma \ref{lem:construction_with_2_chords} shows
$\text{length}(\tilde{\cal{C}})\ge 4l-2$. The conclusion follows.
\end{proof}

The lower bounds on $\gamma_4,\gamma_5,\gamma_7,\gamma_{11}$ are proved by an explicit construction as well.
We start with a general construction method from which these claims easily follow

\begin{lemma}\label{lm:lower_bound_k_better}
Given an $n$-vertex $d$-regular graph $H=(V',E')$ with girth $m$, there is an explicit 
$2nd$-vertex cubic graph $G=(V,E)$ with Petersen Partition $M\sqcup C$, where
the shortest length of a cycle in $G$ with at most $m-1$ chords is $2d$.
\end{lemma}
\begin{proof}
To construct $G$, associate with every vertex v in H a $2d$-cycle $C^{(v)}$.
The $2$-factor $C$ of $G$’s Petersen’s decomposition is the union of $C^{(v)}$ over all $v\in V'$. 
For every edge $v_1v_2\in E'$ we introduce two edges: One between a vertex $v\in C^{(v_1)}$ and a vertex $u\in C^{(v_2)}$. 
The second edge is $u'v'$, where $v'$ the antipode of $v$ in $C^{(v_1)}$ and $u'$ is the antipode of $u$ in $C^{(v_2)}$.
Since girth$(H)=m$, a cycle in $G$ with fewer than $m$ chords can be either one of the $2d$-cycles $C^{(v)}$ or 
it must include some pairs of antipodal vertices, as described above. 
However, in the latter case it must include both chords $uv$ and $u'v'$, and hence 
its length must be at least $2d+2$.
\end{proof}

\begin{corollary}
\[\gamma_4(n)\ge \gamma_5(n) \ge 2(\frac{n}{4})^{1/3}+O(1),\] 
\[\gamma_7(n) \ge 2(\frac{n}{4})^{1/4}+O(1)\] and
\[\gamma_{11}(n) \ge 2(\frac{n}{4})^{1/6}+O(1).\]
\end{corollary}
\begin{proof}
Let $q$ be a prime power, these statements follows from three infinite families of Moore Graphs where $d=q+1$,
and $g=6,8,12$. All these constructions come from projective planes.
For the case where $g=6,d=q+1$ we recall the construction of $H$, the points vs.\ lines graph over the finite field of order $q$. 
This is a bipartite $(q+1)$-regular graph, whose two sides are called $P$ and $L$ for "points" and "lines", with $|P|=|L|=q^2+q+1$. The girth of $H$ is $6$.
As for the cases where $g=8,12$, and $d=q+1$ these graphs are similarly incidence graphs of generalized Quadrangles and generalized Hexagons \cite{biggs_1974}. 

\end{proof}

\begin{corollary}
Let $q$ be a prime power then
\[\gamma_q(n) \ge (n)^{\frac{4}{3q}}\]
\end{corollary}
\begin{proof}
This is based on a family of graphs due to Lazebnik, Ustimenko and Woldar \cite{Lazebnik1995ANS}. For $q$ a prime power,
they construct a $q$-regular graph $G(V,E)$ of order $2q^{k-t+1}$ and girth $g\ge k+5$. 
Here $k\ge1$ is an odd integer, and $t=\lfloor\frac{k+2}{4}\rfloor$. Namely, $G$
is a $q$-regular graph with $n$ vertices and $g\ge \frac{4}{3}\log_q(q-1)\log_{q-1}(n)$ .
\end{proof}

\section{Something on the general girth problem}\label{s:genaral_d}

Some of our results have a bearing on the girth problem for $d$-regular graphs
also for $d>3$. As shown
in \cite{odd_g_cage_are_k_connected} and \cite{even_g_cage_are_k_connected}, every
$(d,g)$-cage is $d$-edge-connected, whence there is no loss in generality in considering only
$(d-1)$-edge-connected $d$-regular graphs $G=(V,E)$. As
shown in \cite{LINIAL198153}, every such graph
has a $2$-factor $C$. In this view, we ask again about short cycles with few chords in $G$,
where a {\em chord} is an edge in $M:=E\setminus C$.
\begin{definition}
$\gamma^d_k(n)$ is the smallest integer $g$ such every $d$-regular $n$-vertex graph with a
given $2$-factor $C$ has a cycle of length $\le g$ with at most $k$ chords.
\end{definition}

We illustrate the connection by proving the following lemma.

\begin{lemma}
$\gamma^d_2(n)\le \sqrt{\frac{2n}{d-2}}$. 
\end{lemma}
\begin{proof}
Again, for any set of vertices $S\subseteq V$, we denote
$$S^*:=\{v\in V|\text{~there is some~} u\in S\text{~such that~}uv\in M \text{~is a chord}\}.$$
Let $A$ be an arc of length $k\le \frac{2n}{d-2}$ then $|A^*|= (d-2)k$. Let $u_1,\dots,u_{k(d-2)}$
be the vertices of $A^*$ in cyclic order. For every $i=1,\dots,k(d-2)$
we define the following $2$-chord cycle. It starts with the chord $u_i,u_i^*$, then proceeds
along $C$ to $u_{i+1}^*$, traverses the chord $u_{i+1},u_{i+1}^*$ and continues along $C$ to $u_i$. 
There are $k(d-2)$ such cycles in total and as in the proof of Lemma \ref{2_chords_uppe_bound},
their total length does not exceed $\le n+\frac{k^2}{2}\cdot (d-2)$.
Consequently the average length of such cycle is bounded from above by 
\[\frac{n}{k(d-2)}+\frac{k}{2}.\]
We take $k=\sqrt{\frac{2n}{d-2}}$ and conclude that at least one cycle has length $\le \sqrt{\frac{2n}{d-2}}$.

\end{proof}

\printbibliography

@article{daynamic_survey:2013,
  title={Dynamic Cage Survey},
  author={Geoffrey Exoo and Robert Jajcay},
  journal={the electronic journal of combinatorics},
  year={2013},
}

@article{Bannai1973OnFM,
  title={On finite Moore graphs},
  author={Eiichi Bannai and Tatsuro Ito},
  journal={Journal of the Faculty of Science, the University of Tokyo. Sect. 1 A, Mathematics},
  year={1973},
  volume={20},
  pages={191-208}
}

@article{damerell_1973,
title={On Moore graphs}, volume={74}, number={2}, journal={Mathematical Proceedings of the Cambridge Philosophical Society}, publisher={Cambridge University Press}, author={Damerell, R. M.}, year={1973}, pages={227–236}}

@article{Petersen_theorem,
  title={The theory of regular graphs},
  author={Julius Petersen},
  journal={Acta Mathematics},
  year={1891},
  volume={15},
  pages={193- 220}
}

@article{Erd1963RegukreGG,
  title={Reguläre Graphen gegebener Taillenweite mit minimaler Knotenzahl},
  author={Paul Erdös and Horst Sachs},
 journal={Mathematics},
  year={1963}
}

@article{girths_Sextet_graphs,
  title={Girths of bipartite sextet graph},
  author={Alfred Weiss},
  journal={Combinatorica },
  year={1984},
  pages={241-245},
}

@article{Sextet_graphs,
  title={The sextet construction of cubic graphs},
  author={N. Biggs and M. Hoare},
  journal={Combinatorica },
  year={1983},
  volume={3},
    pages={153- 165},
}

@article{Lazebnik1995ANS,
  title={A new series of dense graphs of high girth},
  author={Felix Lazebnik and Vasiliy A. Ustimenko and Andrew J. Woldar},
  journal={Bulletin of the American Mathematical Society},
  year={1995},
  volume={32},
  pages={73-79}
}

@book{biggs_1974,
 series={Cambridge Mathematical Library}, 
 title={Algebraic Graph Theory}, author={Biggs, Norman}, year={1974}, collection={Cambridge Mathematical Library}}

@article{LINIAL198153,
title = {On Petersen's graph theorem},
journal = {Discrete Mathematics},
volume = {33},
pages = {53-56},
year = {1981},
author = {Nathan Linial}
}

@article{even_g_cage_are_k_connected,
  title={All (k;g)-cages are k-edge-connected},
  author={Yuqing Lin and Mirka Miller and Christopher A. Rodger},
  journal={Journal of Graph Theory},
  year={2005},
  volume={48},
  pages={219-227}
}

@article{odd_g_cage_are_k_connected,
author = {Wang, Ping and Baoguang, Xu and Wang, Jianfang},
year = {2003},
month = {11},
title = {A Note on the Edge-Connectivity of Cages},
volume = {10},
journal = {Electron J Combin},
}
\end{document}